\documentclass[a4paper,9pt,twoside]{amsart}
\usepackage[margin=10em]{geometry}

\usepackage{amssymb}
\usepackage[hyphens]{url} 
\usepackage{amsmath}
\usepackage{blindtext}
\usepackage{mathtools}
\usepackage{adjustbox}
\usepackage[all,cmtip]{xy}
\usepackage[shortlabels]{enumitem}
\usepackage{multirow}
\usepackage{eucal}
\usepackage{xcolor} 
\usepackage{stackrel}
\usepackage{epigraph}
\usepackage{csquotes}
\usepackage{url}
\usepackage{tikz-cd}
\usepackage{titletoc}
\usepackage{setspace}
\usepackage{chngcntr}
\counterwithin{equation}{section}

\tikzset{
curvarr/.style={
  to path={ -- ([xshift=2ex]\tikztostart.east)
    |- (#1) [near end]\tikztonodes
    -| ([xshift=-2ex]\tikztotarget.west)
    -- (\tikztotarget)}
  }
}

\usepackage{hyperref}
\hypersetup{colorlinks=true}

\usepackage{tikz}
\usetikzlibrary{matrix,arrows,positioning}
\usepackage[bbgreekl]{mathbbol}
\DeclareSymbolFontAlphabet{\mathbb}{AMSb}
\DeclareSymbolFontAlphabet{\mathbbl}{bbold}

\theoremstyle{plain}% default
\newtheorem{thm}{Theorem}[subsection]
\newtheorem*{thm*}{Theorem}

\newtheorem{goal*}{Goal}
\newtheorem{dream*}{Dream}
\newtheorem{lem}[thm]{Lemma}
\newtheorem{prop}[thm]{Proposition}
\newtheorem{cor}[thm]{Corollary}

\theoremstyle{definition}
\newtheorem{dfn}[thm]{Definition}

\newtheorem{exm}[thm]{Example}

\newtheorem{rem}[thm]{Remark}

\theoremstyle{remark}

\newtheorem*{rem*}{Remark}
\newtheorem*{prop*}{Proposition}
 
\newcommand{\R}{\mathbf{R}}

\newcommand{\Hrm}{\mathrm{H}}
\usepackage{stmaryrd}

\newcommand{\GL}{\mathrm{GL}}

\newcommand{\Tw}{\mathrm{Tw}}

\newcommand{\Spec}{\mathrm{Spec}}
\newcommand{\fib}{\mathrm{fib}}
\newcommand{\perf}{\mathrm{perf}}

\newcommand{\Kcal}{\mathcal{K}}

\newcommand{\tr}{\mathrm{tr}}
\newcommand{\Ocal}{\mathcal{O}}

\newcommand{\Lbf}{\mathbf{L}}

\newcommand{\E}{\mathbb{E}}

\newcommand{\Q}{\mathbf{Q}}
\newcommand{\Sfrak}{\mathfrak{S}}

\newcommand{\TC}{\mathrm{TC}}

\newcommand{\Lcal}{\mathcal{L}}

\newcommand{\Scal}{\mathcal{S}}

\newcommand{\C}{\mathbf{C}}

\newcommand{\Sp}{\mathrm{Sp}}

\newcommand{\Sbb}{\mathbb{S}}
\newcommand{\Z}{\mathbf{Z}}

\newcommand{\Art}{\mathrm{Art}}

\newcommand{\ch}{\mathrm{ch}}

\newcommand{\Gal}{\mathrm{Gal}}

\newcommand{\Setcal}{{\mathcal{S}et}}

\newcommand{\Ccal}{\mathcal{C}}
\newcommand{\Mod}{\mathcal{M}\mathrm{od}}

\newcommand{\Hom}{\mathrm{Hom}}

\newcommand{\Tor}{\mathrm{Tor}}

\newcommand{\cyc}{\mathrm{cyc}}

\newcommand{\KU}{\mathrm{KU}}
\newcommand{\et}{\mathrm{\acute{e}t}}
\newcommand{\Ab}{\mathcal{A}b}

\newcommand{\Mcal}{\mathcal{M}}

\newcommand{\cfrak}{\mathfrak{c}}

\newcommand{\acyc}{\mathrm{acyc}}

\newcommand{\Ch}{\mathrm{Ch}}

\newcommand{\idbf}{\mathbf{1}}

\usepackage[backend=bibtex, style=alphabetic]{biblatex}
\title{Homotopy groups of spectra and  $p$-adic $L$-functions over totally real number fields}
\author{Guillem Sala Fernandez}

\begin{document}
\maketitle
\begin{abstract}
The goal of this paper is to illustrate different approaches to understand Euler characteristics in the setting of totally real commutative and non-commutative Iwasawa theory. In addition to this, and in the spirit of \cite{Hess18} and \cite{Mitchell2005K1LocalHT}, we show how these objects interact with homotopy theoretic invariants such as ($K(1)$-local) $K$-theory and Topological Cyclic homology.
\end{abstract}
% \tableofcontents

\section*{Introduction}

\subsection*{Summary of the article and structure} The main motivation for this paper was Hesselholt's expository note \cite{Hess18}. In his article, he argues that if $p$ is an odd number and $F_p$ denotes the $p$-completion of the homotopy fiber of the cyclotomic trace map $\tr:K(\Z)\to\TC(\Z;p)$, then, among others, but more interestingly to us: 
\begin{enumerate}
	\item The value of the Kubota-Leopoldt $p$-adic $L$-function $L_p(\Q,\omega^{-2k}, 1+2k)$ is non-zero if and only if $\pi_{4k+1}(F_p)$ is zero, and
	\item If $L_p(\Q,\omega^{-2k}, 1+2k)\neq 0$, then $\pi_{4k}(F_p)$ and $\pi_{4k-1}(F_p)$ are finite and
	$$\left|L_p(\Q,\omega^{-2k}, 1+2k)\right|_p=\frac{\pi_{4k-1}(F_p)}{\pi_{4k}(F_p)},$$
	where $|-|_p$ denotes the $p$-adic absolute value, 
\end{enumerate}
These results follow from combining the Iwasawa main conjecture for $F=\Q$, together with some well-known results by, for example, Bayer-Neukirch in \cite{Bayer1978} and Schneider in  \cite{Schneider1983}, and more (homotopy) $K$-theoretic results. 

Our first goal -- and the goal of this paper -- was to find an alternative proof of the result, as well as an extension of the result to the case of the Deligne-Ribet $p$-adic $L$-function, that is, to the case where $F$ is a totally real field. As it will be seen in \S\ref{CTRsection}, this was performed successfully, thanks to results by Blumberg and Mandell in \cite{blumberg2015fiber}. 

The second goal was to see if similar results could be reached in the totally real and {\it non-commutative} setting. Unfortunately, it was not possible to find a formula that related the (conjectural) Artin $p$-adic $L$-function with homotopy-theoretic invariants like in the commutative case. It was possible, however, to complete the results involving Euler characteristics in \cite{burns2005leading}, that is, a cohomological description of the $p$-adic norm for certain values of the Artin $p$-adic $L$-function was found. 

The third goal was to see if similar results held in other relevant number theoretic settings, such as the case of imaginary quadratic number fields, or even algebraic function fields, where a $K$-theoretic description of the special values of Katz's $p$-adic $L$-function has already been found. These will be addressed in the second part of this series of two articles. 

Finally, one of the tangential motivations of this paper has been to introduce the number theorist to the field of higher algebra, and the higher algebraist to some of the main topics in Iwasawa theory. For this same reason, a brief but concise introduction to both fields is provided in sections \S\ref{NT background} and \S\ref{htpy theory bg section}. Afterwards, \S\ref{CTRsection} addresses the commutative and totally real case, and \S\ref{NCTR case} addresses the non-commutative counterpart. 
\subsection*{Acknowledgments} The author would like to thank Lars Hesselholt, Jonas McCandless, Andrew Blumberg, Dustin Clausen, Francesc Bars, Francesca Gatti, Armando Gutierrez, Oscar Rivero, Jeanine Van Order, and specially his thesis supervisor Victor Rotger for their help and comments on the article. In addition to this, the project was funded by the European Research Council (ERC) under the European Union's Horizon 2020 research and innovation programme (grant agreement No. 682152).

%%% Here we begin the number theoretic Background
\section{Number theoretic background}\label{NT background}

\subsection{$\Z_p^r$-extensions and the Iwasawa algebra} Let $F$ be a number field, and let $p$ be an odd prime. Recall that, given $r\geq 1$, a {\it $\Z_p^r$-extension} of $F$ is a pro-finite abelian field extension $F_\infty$ of $F$ so that $\Gamma=\Gal(F_\infty|F)\simeq\Z_p^r$. 

\begin{exm}\label{totallyrealzpextension}
Let $F=\Q$, and note that $\Gal(\Q(\mu_{p^{n+1}})|\Q)\simeq\Delta\times\Gamma_n$, where $\Delta$ is a cyclic group of order $p-1$, and $\Gamma_n$ is cyclic of order $p^n$. We can then let $\Q_n$ denote the fixed field of $\Q(\mu_{p^{n+1}})$ by $\Delta$ so that the direct limit $\Q^\cyc=\Q_\infty$ is a $\Z_p$-extension of $\Q$. By the Kronecker-Weber theorem, this is the unique $\Z_p$-extension of $\Q$. Furthermore, given any number field $F$, we can construct the {\it cyclotomic $\Z_p$-extension} of $F$ as $F^\cyc=F\cdot\Q^\cyc$.
\end{exm}

\begin{exm}
Let $F=\Kcal$ be an imaginary quadratic number field. By class field theory, there exists a $\Z_p^2$-extension $\Kcal_\infty$ of $\Kcal$ sitting inside the maximal extension of $\Kcal$ that is unramified outside $p$, which is maximal among the $\Z_p^d$-extensions of $\Kcal$. The Galois group $\Gal(\Kcal|\Q)$ acts on $\Gamma=\Gal(\Kcal_\infty|\Kcal)$ via conjugation, and if we denote the only non-trivial automorphism of $\Gal(\Kcal|\Q)$ by $\sigma$, its action gives us a decomposition
$$\Gamma=\Gamma_+\oplus\Gamma_-$$
where $\sigma\cdot g=\sigma g\sigma^{-1}=g^{\pm 1}$, for each $g\in\Gamma_{\pm}$. We then recover the cyclotomic $\Z_p$-extension taking the fixed field $\Kcal^\cyc$ of $\Kcal_\infty$ by $\Gamma_-$ and obtain the {\it anticyclotomic} $\Z_p$-extension by taking the fixed field of $\Kcal_\infty$ by $\Gamma_+$, which we denote by $\Kcal^\acyc$. Alternatively, let us assume that $p$ splits in $\Ocal_\Kcal$ as $p=\wp_+\wp_-$, then we can construct the $\Z_p$-extensions $\Kcal_\pm\subset\Kcal_\infty$ of $\Kcal$ which are unramified outside $\wp_\pm$.
\end{exm}

Now, let $\Ocal$ be the valuation ring of some finite algebraic extension of $\Q_p$. The main goal of Iwasawa theory is to study the structure of modules over the {\it Iwasawa $\Ocal$-algebra} $\Lambda_\Ocal(\Gamma)=\Ocal\llbracket \Gamma\rrbracket$, which is defined as
$$\Ocal\llbracket \Gamma\rrbracket = \lim_{U\unlhd\Gamma}\Ocal[G/U],$$
where the limit runs through all the open normal subgroups of $G$ with respect to the restriction maps $\Ocal[G/V]\to\Ocal[G/U]$, for $V\leq U$. Note that whenever the base ring $\Ocal$ is understood, we will simply denote $\Lambda(\Gamma)$, or even $\Lambda$. 

Modules over the Iwasawa algebra are also known as {\it Iwasawa modules}, and the next fundamental theorem aims at simplifying the description of such modules. 

\begin{thm}
	{\rm (Iwasawa-Serre, \autocite[Thm 2.3.9]{sharifi})} Let $G$ be a profinite group with $G\simeq\Z_p^r$, and fix a generating set $\{\gamma_i\}_{1\leq i\leq r}$ of $G$. Then, there exists a unique topological isomorphism $\Ocal\llbracket G\rrbracket\to\Ocal\llbracket T_1,\dots, T_r\rrbracket$ sending $\gamma_i-1$ to $T_i$. 
\end{thm}

\begin{rem}
The previous theorem allows us to think of modules over the Iwasawa $\Ocal$-algebra in terms of modules over a formal power series ring. For that reason, we will refer indistinctly to the Iwasawa algebra and its associated formal power series ring, taking into account that this identification carries the extra data of an isomorphism between both algebras. \end{rem}

\subsection{The characteristic polynomial of an Iwasawa module} Now that we know how to think of modules over the Iwasawa algebra, we would like to find invariants that simplify their study. To do so, let $\varpi$ be a uniformizer for the valuation ring $\Ocal$, and let $M$ be a finitely-generated, torsion $\Lambda$-module. It is a classical result that $\Lambda$ is a noetherian Krull domain whose only height 1 primes are either those generated by powers of $\varpi$ or by polynomials $f\in\Ocal[T_1,\dots,T_d]$ that are coprime with $\varpi$, so \autocite[Theorems 4 \& 5,\S VII.4.4]{n1998commutative} can be applied to obtain the existence of a map
\begin{align}\label{structuretheorem}M\to\bigoplus_{n\geq 0}\Lambda/\varpi^{\mu_n}\oplus\bigoplus_{m\geq 0} \Lambda/(f_m^{\lambda_m}),\end{align}
which is a {\it quasi-isomorphism}, meaning that it has finite kernel and cokernel. 

\begin{rem}
Classically, the sum of the powers $\mu_n$ of $\varpi$ give an invariant of $M$, the so-called {\it $\mu$-invariant}, which we denote by $\mu_{\Ocal,\Gamma}(M)$, or simply by $\mu(M)$, whenever the base Iwasawa algebra is understood. This invariant will vanish in the cases that will appear throughout this paper, although this might not always be the case. \end{rem}

\begin{dfn} The {\it characteristic ideal} of $M$ is the product of ideals 
\begin{align}\Ch_{\Lambda}(M)=\varpi^{\mu(M)}\prod_{m=1}^N (f_m^{\lambda_m}).\end{align}
\end{dfn}

\begin{rem} The choice of a generator for $\Ch_\Lambda(M)$ provides a {\it characteristic element} for $M$, and we denote any such choice by $\ch_\Lambda(M)$. 
\end{rem}

\subsection{Twists of Galois modules} From now on, we denote the {\it $p$-adic cyclotomic character} of a number field $F$ by $\chi_{\cyc,F}$, or simply by $\chi_\cyc$ whenever the base field $F$ is understood.  Recall this is the map $\chi_\cyc:G_F\to\Z_p^\times$ defined by $\sigma(\zeta)=\zeta^{\chi_\cyc(\sigma)}$ for all $\zeta\in\mu_{p^\infty}$. One can then define the {\it Tate module} $\Z_p(1)$ to be the $\Z_p[G_F]$-module whose action is given by
$$\sigma\cdot a = \chi_\cyc(\sigma)\cdot a,\qquad \forall\sigma\in G_F, \forall a\in\Z_p.$$
More generally, note that the {\it $n$-th twist} $\Z_p(n)$ of the Tate module is obtained by replacing $\chi_\cyc(\sigma)$ in the above definition by $\chi_\cyc(\sigma)^n$, and with this, given any $\Z_p[G_F]$-module $M$, one may define 
$$M(n)=M\otimes_{\Z_p}\Z_p(n),\qquad \sigma\cdot m =\chi_\cyc(\sigma)^n(\sigma\cdot m),\qquad\forall\sigma\in G_F, \forall m\in M.$$
\begin{exm}
Given a number field $F$ and a choice of a compatible system $\zeta_{p^n}$ of primitive $p^n$-th roots of unity, one obtains isomorphisms of $\Z_p[G_F]$-modules $\Z_p(1)\simeq \varprojlim \mu_{p^n}$ and $ \Q_p/\Z_p(1)\simeq \mu_{p^\infty}$.	
\end{exm}

\begin{rem}\label{cyclotomic char remark} It is a well-known result that the cyclotomic character factors through $\Gal(F(\mu_{p^\infty})|F)$. Thus, given that there is an isomorphism $\Gal(F(\mu_{p^\infty})|F)\simeq\Delta\times\Gamma$, where $\Gamma$ is the Galois group of the cyclotomic $\Z_p$-extension from Example \ref{totallyrealzpextension}, whenever we are given a module $M$ over the Iwasawa algebra $\Lambda$, we can form $M(n)$ by twisting the action of $\Gamma$ by the $\Gamma$-part of $\chi_\cyc$.	
\end{rem}

The ``twist'' provided by the cyclotomic character is the most commonly known, although from the definition one can tell that this can be extended to more general characters. For that reason, we provide the following definition.

\begin{dfn}
Let $\chi:G_F\to\Ocal^\times$ a continuous character, possibly of infinite order, and let $\Ocal(\chi)$ be the free $\Ocal$-module of rank 1 whose action of $G_F$ is given by $\chi$. Then, for any $\Ocal[G_F]$-module $M$, we define $M(\chi)=M\otimes_{\Ocal}\Ocal(\chi)$. 
\end{dfn}

\begin{lem} \label{TorsionPreservation} {\rm (\autocite[\S IV, Lemma 1.2]{rubin2014euler}}) Let $M$ be a finitely generated torsion $\Lambda$-module, and let $\chi:\Gamma\to\Ocal^\times$ be a character. Then, $M(\chi)$ is also a finitely generated torsion $\Lambda$-module. 
\end{lem}

\iffalse{\begin{dfn} Let $\rho:\Gamma\to\GL_d(\Ocal)$ be a continuous representation of $\Gamma$. Then, if we let $T_\rho$ denote a free $\Ocal$-module of rank $d$ that realizes $\rho$, we define the {\it twist} of an $\Ocal\llbracket \Gamma\rrbracket$-module $M$ by $\rho$ by $M(\rho) =M\otimes_{\Z_p}T_\rho$,
which is an $\Ocal\llbracket \Gamma\rrbracket$-module on which $\Gamma$ acts diagonally.
\end{dfn}}\fi 

\subsection{Pontryagin Duality} Recall that one can define an endofunctor on the category of locally compact abelian groups, which is usually denoted by $\mathrm{LCA}$, via the rule $A\mapsto A^\vee:=\Hom_{\mathrm{cts}}(A,\Q/\Z)$, where the image is endowed with the compact-open toplogy. We refer to $A^\vee$ as the {\it Pontryagin dual} of $A$. It is then a well-known result that the contra-variant functor $(-)^\vee:\mathrm{LCA}\to\mathrm{LCA}$ produces an equivalence of categories (cf. \autocite[2.5.2]{sharifi}). 
\begin{rem}
	Note that if $A$ is pro-$p$ or discrete $p$-torsion, then 
	$$A^\vee=\Hom_\mathrm{cts}(A,\Q_p/\Z_p),$$
	and if $A$ is $\Z_p$-finitely generated, then since every $\Z_p$-linear homomorphism is continuous, we have that
	$$A^\vee=\Hom_{\Z_p}(A,\Q_p/\Z_p).$$
\end{rem}

\begin{rem}
If $A$ has the additional structure of a topological $G$-module, then it can be endowed with the continuous action given by
$$(g\cdot f)(a)=f(g^{-1}\cdot a),\qquad\forall g\in G,\forall a\in A,\forall f\in A^\vee.$$	
\end{rem}

\begin{exm} Clearly, one has that $\Z_p^\vee=\Q_p/\Z_p$. Furthermore, note that if we endow $\Z_p(1)^\vee$ with the action from the previous remark, then we obtain $\Z_p(1)^\vee=\Q_p/\Z_p(-1)$. 	
\end{exm}

\begin{lem} {\rm (\autocite[(5)]{howson2002euler})} Let $M$ be a finitely generated $\Lambda(G)$-module. Then, the Pontryagin dual induces a canonical isomorphism between group homology and group-cohomology, namely,
$$\Hrm^n(G,M^\vee)\cong\Hrm_n(G,M)^\vee.   $$
	
\end{lem}

\subsection{Eigenspaces} \label{eigenspaces} Let us consider $\hat{\Delta}$, the group of $p$-adic characters of $\Delta$, where $\Delta$ is a finite abelian group, and let $\Ocal$ be the $\Z_p$-algebra generated by the roots of unity of order dividing the exponent of $\Delta$. Also, given $\psi\in\hat{\Delta}$, we denote the $\Z_p$-algebra generated by the values of $\psi$ by $\Ocal_\psi$. 

Note that $\psi$ extends to a map $\Z_p[\Delta]\to\Ocal_\psi$, which in turn extends to a map $\Ocal[\Delta]\to\Ocal$. In particular, given an $\Ocal[\Delta]$-module, we define $A^\psi:=A\otimes_{\Ocal[\Delta]}\Ocal$ using the previous map, and define
$$e_\psi=\frac{1}{\sharp\Delta}\sum_{\delta\in\Delta}\psi(\delta)\delta^{-1}\in\mathrm{Frac}(\Ocal_\psi)[\Delta].$$
It is then easy to see that $e_\psi\Ocal[\Delta]=e_\psi\Ocal$, and thus we obtain the following result.
\begin{prop}{\rm (\autocite[Prop 2.8.6]{sharifi})} Assume $p$ does not divide $\sharp\Delta$. Then, for any $\Ocal[\Delta]$-module $A$, we have that $A^\psi=e_\psi A$, and there is a decomposition
$$A\cong \bigoplus_{\psi\in\hat{\Delta}}A^\psi.$$
\end{prop}

In the case $\Ocal=\Z_p$, we define $A^{(\psi)}:=A\otimes_{\Z_p[\Delta]}\Ocal_\psi$, and note that in order to be able to define an analogue of $e_\psi$, we need to work over a field containing the images $\psi$. For that reason, we define
$$\tilde{e}_\psi = \frac{1}{\sharp\Delta}\sum_{\delta\in\Delta}\mathrm{Tr}_{\mathrm{Frac}(\Ocal_\psi)|\Q_p}(\psi(\delta))\delta^{-1}\in\Q_p[\Delta],$$
We write $\psi\sim\varphi$ whenever there exists $\sigma\in G_{\Q_p}$ so that $\psi = \sigma\circ\varphi$. Note that under this notation, we have that if $\psi\sim\varphi$, then $\Ocal_\psi = \Ocal_\varphi$. Thus, we can now state the relationship between working with $\Z_p[\Delta]$-modules and $\Ocal[\Delta]$-modules.

\begin{prop}{\rm (\autocite[Lem 2.8.14 \& Prop 2.8.15]{sharifi})} Let $A$ be a $\Z_p[\Delta]$-module.
\begin{enumerate}
\item One has that $A^{(\psi)}\otimes_{\Ocal_\psi}\Ocal\cong (A\otimes_{\Z_p}\Ocal)^\psi$, and if $p$ does not divide $\sharp\Delta$, then
 $$A^{(\psi)}\otimes_{\Z_p}\Ocal\cong \bigoplus_{\varphi\in[\psi]}(A\otimes_{\Z_p}\Ocal)^\varphi.$$
\item If $p$ does not divide $\sharp\Delta$, then 
$$A\cong \bigoplus_{[\psi]\in\hat{\Delta}/\sim}A^{(\psi)}$$	
\end{enumerate}
\end{prop}

\subsection{Twists and characteristic ideals} Our goal now is to see what is the effect of twisting a module on its characteristic ideal. In order to do so, let $\chi:G_F\to\Ocal^\times$ be a character (possibly of infinite order), and define the $\Ocal$-linear isomorphism $\Tw_\chi:\Lambda\to\Lambda$ given by the rule $\gamma\mapsto \chi(\gamma)\gamma$. We have the following result, which is possible due to \ref{TorsionPreservation}. 

\begin{lem}
	{\rm (\autocite[\S VI, Lemma 1.2]{rubin2014euler}} Let $M$ be a finitely generated and torsion $\Lambda$-module. Then, 
	$$\Tw_\chi\left(\Ch_\Lambda(M(\chi))\right)=\Ch_\Lambda(M).$$
\end{lem}

\begin{exm}\label{characteristic power series twist description}
Let us restrict our attention to the case where $F=\Q$. Then, if we let $M$ be a finitely generated torsion $\Lambda$-module, we can conclude from the previous result that, after a choice of a generator $\gamma$ of $\Gamma$, the Iwasawa-Serre isomorphism produces an equality
$$\mathrm{ch_\Lambda}\left(M(-n)\right)(T)\equiv\ch_{\Lambda}(M)(u^n(T+1)-1)\mod\Z_p^\times,$$
which holds for any pair of characteristic elements for $M$ and $M(n)$ respectively, where $u=\chi_\cyc(\gamma)$. 
\end{exm}

\subsection{Euler characteristics of Iwasawa modules} Finally, we introduce Euler characteristics of Iwasawa modules. 

\begin{dfn} Let $M$ be a finitely-generated $\Lambda(G)$-module, for $G$ the Galois group of a $\Z_p^d$-extension. Then, $M$ is said to have {\it finite $G$-Euler characteristic} if all the (group) homology groups 	$\Hrm_i(G,M)$ are finite for all $0\leq i\leq d$. Furthermore, if $M$ has finite $G$-Euler charactersistic, then the {\it $G$-Euler characteristic of $M$} is defined as the product
$$\mathrm{EC}(G,M):=\prod_{0\leq i\leq d}\sharp\Hrm_i(G,M)^{(-1)^i}.$$
\end{dfn}

Note that given a finitely generated $\Lambda=\Lambda(\Gamma)$-module, one can define analogously its $\Gamma$-Euler characteristic, and the following result illustrates the relation between both modules. The result can be found in, e.g, \autocite[(31), p.14]{coates2012noncommutative}

\begin{thm}
	Given a finitely generated $\Lambda(G)$-module $M$, $M$ has finite $G$-Euler characteristic if and only if the $\Gamma$-modules $\Hrm_i(\Gamma,M)$ have finite $\Gamma$-Euler characteristic for all $0\leq i< d$. Whenever any of the two equivalent conditions holds, one also has the equality
	$$\mathrm{EC}(G,M)=\prod_{0\leq i<d}\mathrm{EC}(\Gamma,\Hrm_i(\Gamma,M))^{(-1)^i}.$$
\end{thm}

%%% Here we begin the Homotopy theoretic background
\newpage
\section{Background in homotopy theory}\label{htpy theory bg section}

The following section will serve the purpose of illustrating the subject of stable homotopy theory using the language of $\infty$-categories. By no means this intends to be a precise introduction, our main goal here is to allow the reader that is not familiarized with the language to follow the rest of this paper up to certain technicalities. The main references are \cite{lurie2009higher} and \cite{lurie2016higher}.

\subsection{The language of $\infty$-categories} For our purposes in this paper, we can think of an {\it $\infty$-category} as a category that is weakly enriched over spaces, in the sense that there is a {\it space} of maps between every pair of objects, and the composition law is defined only up to {\it coherent} homotopy. Homotopy coherence is the problem one faces when trying to lift diagrams in the homotopy category of spaces to the ordinary category of spaces. It is easy to see that the existence of such a lift is equivalent to the existence of certain higher-degree homotopies, and this is precisely the problem that $\infty$-categories tackle, as we are asking their sets of maps to be spaces, and therefore contain homotopical information in all degrees.  Although we are asking for the composition law to be only unique up to coherent homotopy, which is a much stronger condition than being unique up to homotopy, this is still sufficient to provide a robust theory of $\infty$-categories, in which all constructions for $1$-categories (functors, limits, colimits...) can also be performed. 

\begin{rem}
One can also look at the weaker concept of an {\it $\infty$-groupoid}, whose main difference with $\infty$-categories is that all maps admit an inverse. As explained in \autocite[Example 1.1.1.4]{lurie2009higher}, one can think of $\infty$-groupoids as categories behaving as spaces, and from it one defines the $\infty$-category of spaces $\Scal$, which is the $\infty$-categorical analogue of the classical category of spaces.
\end{rem}

\iffalse Let $X$ be a topological space and let $\pi_{\leq 1}(X)$ be the category whose objects are the points of $X$ and whose maps are paths between points. Note that one can actually find homotopies between paths thus giving rise to maps between maps -- also called 2-maps. Constructing homotopies between homotopies leads to the {\it fundamental groupoid} $\pi_{\leq\infty} X$. This should give us a starting picture to motivate what an {\it $\infty$-category} is: a category enriched with morphisms of higher degree in which $(k-1)$-maps admit $k$-maps between them. The problem with this approach is that keeping track of the composition laws becomes a tedious task in high degrees, and so one uses the language of simplicial sets, although we will not go deeper into the theory.

Furthermore, just as the composition law in $\pi_{\leq\infty} X$ is unique only up to homotopy, the same happens for $\infty$-categories. Luckily, this is still sufficient to provide a robust theory of $\infty$-categories, in which all constructions for $1$-categories (functors, limits, colimits...) can also be performed. Thus, the reader should think of $\infty$-categories as ``homotopically enriched categories". \fi

\subsection{Stable $\infty$-categories} We will be working with {\it stable $\infty$-categories} most of the time. These are pointed\footnote{An $\infty$-category is said to be pointed if it admits a zero object, that is, an object that is both initial and final.} $\infty$-categories with finite limits, and whose loop functor gives an equivalence\footnote{The condition of being pointed and having finite limits is enough to construct a loop space functor $\Omega$ and a suspension functor $\Sigma$. Furthermore, the stability condition implies that the adjunction $(\Sigma,\Omega)$ gives an equivalence.}. In addition, there is a precise notion of fiber and cofiber sequences just as in classical homotopy theory. By \autocite[Theorem 1.1.2.14]{lurie2016higher}, the homotopy category of every stable $\infty$-category admits a triangulated structure, which illustrates another reason why one would be interested in working the $\infty$-categorical setting: the complexity of triangulated categories could be justified by its being a ``decategorification" of a structure that is natural in higher category theory. 
 
 Moreover, the homotopy category of a stable $\infty$-category $\Ccal$ can be endowed with a $t$-structure, although it might not be unique. Roughly, a $t$-structure is the additional data in the homotopy category of $\Ccal$ that axiomatizes the phenomenon of truncation of complexes and spaces and, therefore, taking cohomology or homotopy groups. From a $t$-structure, one can obtain the {\it heart} $\Ccal^\heartsuit$ of $\Ccal$. This is an abelian subcategory of the homotopy category of $\Ccal$ which, in the case of the derived $\infty$-category\footnote{The derived $\infty$-category admits a canonical $t$-structure.}, can be regarded as those complexes concentrated in degree 0. A consequence of this construction is that fiber sequences in $\Ccal$ induce long exact sequences in $\Ccal^\heartsuit$ after taking homotopy groups.

\subsection{The $\infty$-category of Spectra} For our purposes, we will think of the $\infty$-category of spectra, which is denoted by $\Sp$, as the higher analogue of the 1-category of abelian groups. More precisely, there exist adjoint functors $\Omega^\infty$ and $\Sigma^\infty_+$ giving rise to the following diagram
$$
\xymatrix@C=8em{
\Ab \ar@{->}@/_/[r] \ar@/^/[d]^{\mathrm{forget}} & \Sp \ar@/_/[l]_{\pi_0} \ar@/^/[d]^{\Omega^\infty} \\
\Setcal \ar@/^/[u]^{\Z(-)} \ar@{->}@/_/[r] & \Scal. \ar@/^/[u]^{\Sigma^\infty_+} \ar@/_/[l]_{\pi_0}
}
$$
\begin{rem} The functor $\pi_0:\Sp\to\Ab$ comes from the fact that $\Sp$ admits a canonical $t$-structure, and its heart can be naturally identified with (the nerve of) the 1-category of abelian groups. At the same time, this implies that any fibration in $\Sp$ gives rise to a long exact sequence of abelian groups just as in classical homotopy theory. However, note that in this case one can take negative homotopy groups.\end{rem}

An alternative way to visualize $\Sp$ is given by Brown's representability theorem \autocite[Theorem 1.4.1.2]{lurie2016higher}, which roughly states that any generalized cohomology theory is representable by a spectrum. Thus, $\Sp$ can be seen as the $\infty$-category containing all cohomology theories, although we warn the reader that the embedding is not always faithful. Furthermore, given a space $X$, one can define $E^n(X)=[\Omega^n E,\Sigma^\infty_+X]$, and check that the assignment $X\mapsto E^n(X)$ gives a generalized cohomology theory. In fact, if $X$ is also a spectrum, one can define $E^n(X)=[\Omega^nE,X]$.
\begin{exm}
We give a list of some spectra that we will use during this article.
\begin{enumerate}[nosep]
\item A central object in the study of spectra is the {\it sphere spectrum}, which is the $\infty$-suspension of the point, $\Sbb=\Sigma_+^\infty(\ast)$. This object is sometimes referred to as the {\it stable sphere}, and the reason is that its homotopy groups agree with the stable homotopy groups of spheres.
\item As in the classical case, given an abelian group $G$, one can always construct its {\it Eilenberg-MacLane spectrum}, denoted by $HG$, although when the context of spectra is understood, we will simply denote it by $G$. It is characterized by the fact that its homotopy groups are concentrated in degree 0.
\item We will also use the {\it Moore spectrum} $SG$, which is characterized by having $\pi_iSG=0$ for $i<0$, $\pi_0SG=G$, and $\pi_i(SG\otimes H\Z)=H_i(SG;\Z)=0$ for $i>0$. The importance of this spectrum will become visible in the Bousfield localization section below, where we will work with {\it non-connective spectra}, that is, spectra with homotopical information in negative degrees.
\item Given a spectrum $X$, one can associate to it its {\it $K$-theory spectrum} $K(X)$. In particular, when $X=HR$ for some ring $R$, $\pi_n K(R)\simeq K_n(R)$, where the right hand side denotes the classical $n$-th K-theory group.
\end{enumerate}
\end{exm}

Finally, just as the category of abelian groups admits a symmetric monoidal category structure, one can endow $\Sp$ with a symmetric monoidal $\infty$-category structure given by the smash product $\otimes$ and whose unit is the sphere spectrum. This provides the possibility of forming algebraic structures over $\Sp$, whence the name {\it Higher Algebra}. 
\begin{center}
\begin{tabular}{c|c}
Ordinary Algebra & $\infty$-Categorical Algebra\\
\hline Set & Space \\
Abelian group & Spectrum \\
Tensor product of abelian groups & Smash product of spectra\\
Associative ring & $\E_1$-ring \\
Commutative ring & $\E_\infty$-ring
\end{tabular}\footnote{This table was obtained from \autocite[Chapter 7]{lurie2016higher}.}
\end{center}

\subsection{Bousfield localization of spectra and Iwasawa theory} Let us fix a spectrum $E\in\Sp$. Another spectrum $X$ is said to be {\it $E$-acyclic} if $X\otimes E\simeq 0$. On the other hand, $X$ is said to be {\it $E$-local} if $Y\to X$ is null-homotopic (i.e, factors through a zero object) whenever $Y$ is acyclic. Then, there exists a functor $L_E$ from $\Sp$ called {\it Bousfield localization}, which is characterized up to equivalence by the following two properties, namely that  $L_E X$ is $E$-local and that there is a natural map $X\to L_E X$ and its fiber is $E$-acyclic. We then define the {\it $p$-adic completion} $X_p^\wedge$ of a spectrum $X$ to be the Bousfield localization of $X$ with respect to the Moore spectrum $S\Z/p$. 

Now, let $X$ be a space, and let $\KU^\ast(X)$ be the ring of (periodic) complex $K$-theory. Since $\KU^\ast$ defines a generalized cohomology theory, it corresponds to a spectrum $\KU$ by Brown's theorem. We refer to $\KU$ as the {\it complex $K$-theory spectrum}. Finally, we define the {\it $K(1)$-local $K$-theory spectrum} as $L_{K(1)}K=\KU_p^\wedge$.

\begin{rem}
In fact, the $K(1)$-local adjective comes from the first Morava $K$-theory spectrum, which satisfies that $L_{K(1)}K(X)\simeq\KU(X)_p^\wedge$. In general, one can construct the $n$-th Morava $K$-theory spectrum $K(n)$, which is a central tool in stable homotopy theory.
\end{rem}

As explained in \cite{Mitchell2005K1LocalHT}, $K(1)$-local $K$-theory already has a direct relation with Iwasawa theory: One can identify the group of Adams operations on $K(1)$-local $K$-theory with $\Gal(\Q_\infty|\Q)$ via the cyclotomic character and obtain $(L_{K(1)}K)^\ast(L_{K(1)}K)\simeq\Lambda(\Gamma)[\Delta]$.

\newpage
%%% Here begins the Commutative TR case 
\section{The commutative and totally real case}\label{CTRsection}

From now on, let us fix an odd prime $p$ and let $F$ be a totally real number field. Let $\Sigma$ denote a finite set of primes of $F$ containing all primes dividing $p$ and let $F_\Sigma$ denote the maximal extension of $F$ which is unramified outside $\Sigma$ and the infinite places of $F$. Also, given an extension $L$ of $F$ inside $F_\Sigma$, we denote the maximal abelian $p$-extension of $L$ contained inside $F_\Sigma$ by $M(L)$, and its Galois group over $L$ by $X(L)=\Gal(M(L)|L)$. 
\begin{dfn}\label{tradmissible}
 A Galois extension $F_\infty|F$ is said to be a {\it totally real admissible $p$-adic Lie extension} if
 \begin{enumerate}[nosep]
 \item $F_\infty$ is totally real.
 \item The Galois group of $F_\infty|F$ is a $p$-adic Lie group, that is, a profinite topological group of finite rank containing an open pro-$p$ subgroup.
 \item $F_\infty|F$ is unramified outside $\Sigma$.
 \item $F_\infty$ contains $F^\cyc$, the maximal $\Z_p$-extension of $F$ contained in $F(\mu_{p^\infty})$.	
 \end{enumerate}
 \end{dfn}
 Finally, and following \autocite[p. 6]{coates2012noncommutative}, we will work under the assumption that $G=\Gal(F_\infty|F)$ is an abelian $p$-adic Lie group of dimension equal to one -- which, according to Leopoldt's conjecture should always be the case --, and denote  $H=\Gal(F_\infty|F^\cyc)$, which is a finite group under our running hypotheses.

\subsection{The main conjecture} Let us start by noting that if we denote $\Gamma=\Gal(F^\cyc|F)$,  $G$ can be written as $G =H\times\Gamma$. If we denote by $K$ the fixed field of $\Gamma$, which satisfies that $K\cap F^\cyc=F$, then $F_\infty$ can be written as the composite field $K F^\cyc$. Moreover, we will refer to the $p$-adic cyclotomic character $\chi_\cyc:\Gal(F(\mu_{p^\infty})|F)\to\Z_p^\times$ in a slightly unconventional way, namely as a character of $G$ that is trivial on $H$, given that $\Gamma$ is a subgroup of $\Gal(F(\mu_{p^\infty})|F)$ and thus we can look at the restriction of $\chi_\cyc$ to $\Gamma$, as it was commented in \ref{cyclotomic char remark}. 

Following the notations from \ref{eigenspaces}, let $\psi\in\hat{H}\setminus\{\idbf\}$ be of {\it type S}, that is, so that $F_\psi$, the field cut out by $\ker(\psi)$, or perhaps more precisely, the minimal extension of $F$ through which $\psi$ factors, is totally real and $F_\psi\cap F^\cyc=F$. For simplicity, we will also assume that the order of $\psi$ is coprime to $p$. Furthermore, fix a topological generator $\gamma\in\Gamma$, whose image through the cyclotomic character $\chi_\cyc(\gamma)$ we denote by $u$,

Now, let $L^\Sigma(\psi,s)$ denote the $L$-function with Euler factors at $\Sigma$ in the Euler product expression removed, whose rationality is justified in \autocite[(1.2)]{Wiles90} by using a result by Klingen-Siegel. We then have the following theorem.
\begin{thm}
\label{DR}    {\rm (Deligne-Ribet \cite{Deligne1980}, Cassou-Nogu\`es \cite{Cassou1979}, Barsky \cite{GAU_1977-1978__5__A9_0})} Let $n\geq 1$ be an integer with $n\equiv 0\mod[F(\mu_p):F]$. There exists a unique formal power series $L_p(\psi)\in\Ocal_\psi\llbracket T\rrbracket$ such that
$$L_p(\psi,u^n-1)=L^\Sigma(\psi,1-n).$$
\end{thm}

\begin{rem}
As explained in \cite{Wiles90}, the theorem can also be stated for $\idbf$ and characters of type $W$, that is, characters in which $F_\psi\subseteq F^\cyc$, multiplying the left hand side by $(\psi(\gamma)u^n-1)^{-1}$ but for simplicity we only focus on the type $S$ case. 
\end{rem}

\begin{rem}
As explained in \autocite[p. 7]{coates2012noncommutative} the theorem can also be stated regarding $L_p(\psi)$ as a {\it pseudo-measure} on $G$, so that if $n>0$ is an integer with $n\equiv 0\mod[F(\mu_p):F]$, 
$$L_p(\psi,\chi_\cyc^n):=\int_G \left(\psi\otimes\chi_\cyc^n\right) dL_p(\psi)=L^\Sigma(\psi,1-n).$$
This is why we will denote $L_p(\psi,\chi_\cyc^n)$ and $L_p(\psi,u^n-1)$ indistinctly.
\end{rem}

For the rest of this section, let us fix $\psi$ of type $S$; denote $X=X(F_\infty)$, and let $X_\psi=e_\psi(X\otimes_{\Z_p}\Ocal_\psi)$. By a classical theorem of Iwasawa, $X_\psi$ is finitely generated and torsion as an $\Ocal_\psi\llbracket\Gamma\rrbracket$-module, and therefore we can look at a characteristic power series $C_\psi:=\ch(X_\psi)\in\Ocal_\psi\llbracket T\rrbracket$ for $X_\psi$. It is easy to check that 
\begin{align}\label{DeterminantFormula}C_\psi(T)\Ocal_\psi\llbracket T\rrbracket =\varpi_\psi^{\mu_\psi}\det\left(T-(\gamma-1)|X_\psi\otimes F_\psi\right)\Ocal_\psi\llbracket T\rrbracket ,\end{align}
where $\varpi_\psi$ is any fixed local parameter for $\Ocal_\psi$, and $\mu_\psi$ is the algebraic $\mu$-invariant, which by the work of Wiles agrees with the analytic invariant obtained after applying the Weierstrass preparation theorem to $L_p(\psi,T)$ above. With this, one can state the algebraic part of the IMC as follows.

\begin{thm}
{\rm (IMC for totally real fields)} \label{AbelianIMCtotallyreal} For all $\psi\in\hat{H}$ of order prime to $p$, 
$$L_p(\psi,T)\Ocal_\psi\llbracket T\rrbracket = C_\psi(T)\Ocal_\psi\llbracket T\rrbracket.$$
\end{thm}

%Combining Theorems \ref{DR} and \ref{AbelianIMCtotallyreal}, we deduce that $L^\Sigma(\psi,1-n)$ and $C_\psi(u^n-1)$ differ by just a unit of $\Ocal_\psi$. 

\subsection{The group-cohomological Euler characteristic formula} Now, let us note that $G$ acts on $\Z_p(n)$ as long as $n\equiv 0\mod[F(\mu_p):F]$, so let $n$ be an integer -- positive or negative -- of that form. The determinant part in equation (\ref{DeterminantFormula}) admits a cohomological interpretation by the work of Bayer-Neukirch \autocite[Theorem 6.1]{Bayer1978} and Schneider \autocite[Theorem \S4.4]{Schneider1983}. Namely, the determinant does not vanish at $T=0$ if and only if the $\Gamma$-invariants (or equivalently, the $\Gamma$-coinvariants) of $X_\psi$ are finite dimensional. When one of these three equivalent conditions holds, one has
\begin{align}\label{AbelianEulerCharFla}|C_\psi(u^{n}-1)|_p=\frac{\sharp \Hrm^0\left(\Gamma,X_\psi(-n)\right)}{\sharp \Hrm^1\left(\Gamma,X_\psi(-n)\right)},\end{align}
where the $(-n)$-th Tate twist of $X_\psi$ appearing in the right hand side follows from \ref{characteristic power series twist description}.  

Now, combining \ref{DR}, \ref{AbelianIMCtotallyreal} and (\ref{AbelianEulerCharFla}), we obtain the following well-known result.

\begin{prop}\label{TRgroupcohomologyformula}
Let $n$ be an integer with $n\equiv 0\mod[F(\mu_p):F]$. Then, 
$$|L_{p}(\psi,\chi_\cyc^n)|_p=\frac{\sharp \Hrm^0\left(\Gamma,X_\psi(-n)\right)}{\sharp \Hrm^1\left(\Gamma,X_\psi(-n)\right)}.$$
\end{prop}

\begin{rem}
In particular, note that when $n>0$ is an integer such that $n\equiv 0\mod[F(\mu_p):F]$, we have a group-cohomological interpretation of the $p$-adic valuation of the $L$-function, that is,
$$\left|L^\Sigma(\psi,1-n)\right|_p=\frac{\sharp \Hrm^0\left(\Gamma,X_\psi(-n)\right)}{\sharp \Hrm^1\left(\Gamma,X_\psi(-n)\right)}.$$
\end{rem}

\subsection{The \'etale-cohomological Euler characteristic formula} \label{etaleeulerchar} Let $n$ be an integer as in \ref{TRgroupcohomologyformula}. We can now use Shapiro's lemma to extend the formula above as follows
\begin{align}\label{eqn1}
    |L_p(\psi,\chi_\cyc^n)|_p=\frac{\# \Hrm^0\left(\Gamma, X_\psi(-n) \right)}{\# \Hrm^1\left(\Gamma,  X_\psi(-n) \right)}=\frac{\# \Hrm^0\left(G,  X(-n)\right)}{\# \Hrm^1\left(G,  X(-n)\right)}.
\end{align}
Let us now recall that we can identify $X$ with $\Hrm^1_\et(\Ocal_{F_\infty}[1/p];\Q_p/\Z_p)^\vee$. This is due to the fact that since $X$ is the Galois group of the maximal abelian $p$-extension of $F^\cyc$ that is unramified outside $p$, it can be identified with the maximal abelian pro-$p$ factor group of the algebraic fundamental group of $\Spec(\Ocal_{F_\infty}[1/p])$. Taking this into account, 
\begin{align*} \Hrm^\nu(G,X(-n))&=\Hrm^\nu\left(G,\Hrm^1_\et(\Ocal_{F_\infty}[1/p];\Q_p/\Z_p(n))^\vee\right)\\
&=\Hrm^{1-\nu}\left(G,\Hrm^1_\et(\Ocal_{F_\infty}[1/p];\Q_p/\Z_p(n))\right)^\vee,\qquad \nu=0,1.\end{align*}
where the first equality follows from the fact that the Pontryagin dual inverts the action of the cyclotomic character, and the second equality follows from the proof of \autocite[Lemma \S I.2.4]{Schneider1983}. Combining these equalities with \ref{eqn1}, we obtain
$$|L_p(\psi,\chi_\cyc^n)|_p=\frac{\# \Hrm^1\left(G, \Hrm^1_\et\left(\Ocal_{F_\infty}[1/p],\Q_p/\Z_p(n)\right)\right)}{\# \Hrm^0\left(G, \Hrm^1_\et\left(\Ocal_{F_\infty}[1/p],\Q_p/\Z_p(n)\right)\right)}.$$
Following \autocite[Lemma \S II.4.1]{Schneider1983} we can use the Hochschild-Serre spectral sequence 
$$E^{i,j}_2=\Hrm^i\left(G,\Hrm^j_\et\left(\Ocal_{F_\infty}[1/p],\Q_p/\Z_p(n)\right)\right)\Rightarrow \Hrm^{i+j}_\et\left(\Ocal_F[1/p],\Q_p/\Z_p(n)\right).$$
Schneider studies the groups appearing in this spectral sequence and observes that it degenerates in the second page, and from there he deduces that, since the weak Leopoldt conjecture is proven in the case of totally real number fields, and that is equivalent to $\Hrm^2_\et(\Ocal_{F_\psi}[1/p],\Q_p/\Z_p)=0$, one has the following classical result.

\begin{thm}\label{etale_formula}
	{\rm (\autocite[Theorem 6.1]{Bayer1978},\autocite[Theorem \S4.4]{Schneider1983})}\label{BayerNeukirch} Let $n$ be an integer with $n\equiv 0\mod [F(\mu_p):F]$. Then, the following equality holds
	$$\left|L_p(\psi,\chi_\cyc^n)\right|_p=\frac{\#\Hrm^0_\et\left(\Ocal_F[1/p],\Q_p/\Z_p(n)\right)}{\#\Hrm^1_\et\left(\Ocal_F[1/p],\Q_p/\Z_p(n)\right)}.$$
\end{thm}

\subsection{($K(1)$-local) $K$-theory and \'etale descent} \label{Ktheoandetalesection} Following the same notation as in the previous sections, as well as the one introduced in \S\ref{htpy theory bg section}, we can state the following theorem. 
\begin{thm}
    {\rm (Blumberg-Mandell, \autocite[Theorem 1 \& Corollary 1]{blumberg2015fiber})}\label{BMTheo} The homotopy fiber $\fib(\kappa)$ of the completion map 
    $$\kappa: L_{K(1)}K(\Ocal_F[1/p])\to \prod_{v\in \Sigma}L_{K(1)}K(F_v)$$
    in $K(1)$-local algebraic $K$-theory is weakly equivalent to $\Omega I_{\Z_p}L_{K(1)}K(\Ocal_F[1/p])$, where $I_{\Z_p}$ denotes the Anderson duality functor for $p$-local spectra. Furthermore, there is a canonical map of spectra from the homotopy fiber of the $p$-adic cyclotomic trace map $$\tr_p^\wedge:K(\Ocal_F;\Z_p)\to\TC(\Ocal_F;\Z_p)$$
    to $\Omega I_{\Z_p}L_{K(1)}K(\Ocal_F)$ that induces an isomorphism of homotopy groups on degrees $\geq 2$. 
\end{thm}

\begin{rem}
Note that, thanks to \autocite[Theorem 1.1]{bhatt2020remarks}, we know that $L_{K(1)}K(\Ocal_F[1/p])\simeq L_{K(1)}K(\Ocal_F)$, which allows us to identify both fibers in the specified degrees. 
\end{rem}

We would like to relate the \'etale cohomology groups showing up in Theorem \ref{BayerNeukirch} and ($K(1)$-local) $K$-theory. To do so, we use the Thomason descent spectral sequence which, given a ring $R$ that we assume to be either $\Ocal_F[1/p]$ or $F_v$ for any $v\in \Sigma$, has the following form
$$E^{p,q}_2=\Hrm^p_\et(R,\Z_p(-q/2))\Rightarrow\pi_{-p-q}L_{K(1)}K(R),$$
where the groups on the left are assumed to be zero whenever $q$ is not divisible by $2$. 

Using Thomason's descent spectral sequence as in \cite{blumberg2015fiber} or \cite{dwyer1998k}, if we denote the Moore spectrum associated to $\Q_p/\Z_p$ by $M(\Q_p/\Z_p)$, we have the following table.  
\begin{center}
\begin{tabular}{| c |c| c |}
\hline
 $X$ & $\pi_{2k+1}X$ & $\pi_{2k}X$ \\ 
 \hline
 \multirow{3}{*}{$L_{K(1)}K(R)$} &  &  $\Hrm^0_\et(R,\Z_p(k))$ \\
    & $\Hrm^1_\et(R,\Z_p(k+1))$ & $\oplus$ \\ & & $\Hrm^2_\et(R,\Z_p(k+1))$ \\
 \hline
  \multirow{3}{*}{$L_{K(1)}\left(K(R)\otimes M_{\Q_p/\Z_p}\right)$} &  &  $\Hrm^0_\et(R,\Q_p/\Z_p(k))$ \\
    & $\Hrm^1_\et(R,\Q_p/\Z_p(k+1))$ & $\oplus$ \\ & & $\Hrm^2_\et(R,\Q_p/\Z_p(k+1))$ \\
    \hline
\end{tabular}
\end{center}

\begin{rem} As explained in Blumberg-Mandell's article, one can use \ref{BMTheo} to identify (non-canonically) the {\it Poitou-Tate exact sequence} 

\begin{center}
\adjustbox{scale=0.97,center}{
\begin{tikzcd} 
0\to \Hrm_\et^0(\Ocal_F[1/p],\Z_p(k)) \arrow[r,"\Pi_0"]
& \prod_{v\in \Sigma}\Hrm^0_\et(F_v,\Z_p(k))\arrow[r]\arrow[d, phantom, ""{coordinate, name=W}]
& \Hrm^2_\et(\Ocal_F[1/p],\Q_p/\Z_p(1-k))^\vee 
\arrow[dll,rounded corners=8pt,curvarr=W] 
\\
\Hrm_\et^1(\Ocal_F[1/p],\Z_p(k)) \arrow[r,"\Pi_1"]
& \prod_{v\in \Sigma}\Hrm^1_\et(F_v,\Z_p(k))\arrow[r] \arrow[d, phantom, ""{coordinate, name=W}]
& \Hrm^1_\et(\Ocal_F[1/p],\Q_p/\Z_p(1-k))^\vee
\arrow[dll,rounded corners=8pt,curvarr=W] 
\\
\Hrm_\et^2(\Ocal_F[1/p],\Z_p(k)) \arrow[r,"\Pi_2"]
& \prod_{v\in \Sigma}\Hrm^2_\et(F_v,\Z_p(k))\arrow[r] 
& \Hrm^0_\et(\Ocal_F[1/p],\Q_p/\Z_p(1-k))^\vee\to 0,
\end{tikzcd}}\end{center}
with the long exact sequence induced by the taking homotopy groups of the fibration in $K(1)$-local $K$-theory
\begin{center}
\begin{tikzcd}
\cdots\to \pi_l\fib(\kappa) \arrow[r]
& \pi_l L_{K(1)}K(\Ocal_F[1/p])\arrow[r]\arrow[d, phantom, ""{coordinate, name=W}]
& \prod_{v\in \Sigma}\pi_l L_{K(1)}K(F_v)
\arrow[dll,rounded corners=8pt,curvarr=W] 
\\
\pi_{l-1}\fib(\kappa) \arrow[r]
&  \pi_{l-1}L_{K(1)}K(\Ocal_F[1/p])\arrow[r] 
& \prod_{v\in \Sigma}L_{K(1)}K(F_v)\to\cdots. 
\end{tikzcd}
\end{center}
This is due to Anderson duality, which provides an isomorphism
\begin{align}\label{AndersonDuality}
    \left(\pi_{-1-l}\left(L_{K(1)}K(\Ocal_F[1/p])\otimes M_{\Q_p/\Z_p}\right)\right)^\vee\simeq \pi_{l+1}( I_{\Z_p}L_{K(1)}K(\Ocal_F[1/p])).\end{align}
    The second term is isomorphic to $\pi_l\fib(\kappa)$ by Theorem \ref{BMTheo}, and the first term can be compared with the Pontryagin dual of the corresponding \'etale cohomology group in the table above, thus giving the identification between both exact sequences. 
\end{rem}

\subsection{The $K(1)$-local $K_\ast$-theoretic description of the special values} We can move on to applying all the machinery we have introduced to give another formula for the special values of the $p$-adic $L$-function, and the results we provide here are an extension, as well as an alternative proof of a result by Hesselholt, who studied the case $F=\Q$ in \cite{Hess18}. 

\begin{cor}\label{hessgeneralization}
Let $n$ be such that $n\equiv 0\mod[F(\mu_p):F]$. Then, 
	$$\left|L_p(\psi,\chi_\cyc^n)\right|_p=\frac{\#\pi_{-1-2n}\fib(\kappa)}{\#\pi_{-2n}\fib(\kappa)}.$$
	Furthermore, if $n \leq -1$,
	$$\left|L_p(\psi,\chi_\cyc^n)\right|_p=\frac{\#\pi_{-1-2n}\fib\left(\tr_p^\wedge\right)}{\#\pi_{-2n}\fib\left(\tr_p^\wedge\right)}$$
\end{cor}

\begin{proof}
    Let us start by observing that 
    \begin{align*}
        \Hrm^0_\et(\Ocal_{F}[1/p],\Q_p/\Z_p(n))^\vee&\simeq\pi_{2n}L_{K(1)}\left(K(\Ocal_{F}[1/p])\otimes M_{\Q_p/\Z_p}\right)^\vee \\
        &\simeq \pi_{-1-2n}\Omega I_{\Z_p}L_{K(1)}K(\Ocal_{F}[1/p])\\
        &\simeq \pi_{-1-2n}\fib(\kappa),
    \end{align*}
    where the middle equality follows from (\ref{AndersonDuality}) above. 
    
    On the other hand, and as mentioned in the previous section, the weak Leopoldt conjecture is true in the case of totally real and imaginary quadratic fields, and that is equivalent to the vanishing of $\Hrm^2_\et(\Ocal_F[1/p],\Q_p/\Z_p)$. Thanks to this, if we use the table above again, we obtain that
	\begin{align*} \Hrm^1_\et(\Ocal_F[1/p],\Q_p/\Z_p(n))^\vee &\simeq \pi_{2n-1}L_{K(1)}\left(K(\Ocal_F[1/p])\otimes M_{\Q_p/\Z_p}\right)^\vee\\
&\simeq \pi_{-2n}\fib(\kappa),\end{align*}
	where the last equivalence is obtained by mimicking the argument above. The degrees on which the homotopy groups of $\fib(\kappa)$ and $\fib(\tr_p^\wedge)$ come from Theorem \ref{BMTheo} above.
\end{proof}

We can also use the $K(1)$-local $K$-theoretic Poitou-Tate duality of Blumberg-Mandell to give a description of the $p$-adic valuation of the special values without using the fiber.

\begin{cor}\label{corollaryk1}
	Let $n$ be an integer with $n\equiv 0\mod[F(\mu_p):F]$. Then, 
	$$
	|L_p(\psi,\chi_\cyc^n)|_p=\frac{\#\pi_{-2n+1}L_{K(1)}K(\Ocal_F[1/p])}{\#\pi_{-2n}L_{K(1)}K(\Ocal_F[1/p])}\cdot\prod_{v\in\Sigma}\frac{\#\pi_{-2n}L_{K(1)}K(F_v)}{\#\pi_{-2n+1}L_{K(1)}K(F_v)}
	$$
\end{cor}

\begin{proof}
We can use the table above together with the fact that the \'etale cohomology group $\Hrm^2_\et(\Ocal_F[1/p],\Q_p/\Z_p)^\vee$ must vanish by the weak Leopoldt conjecture, to conclude that the classical long exact sequence on homotopy groups can be split as follows \begin{center}
\adjustbox{center}{
\begin{tikzcd}
0\to \pi_{1-2k}L_{K(1)}K(\Ocal_F[1/p])\arrow[r,"\Pi_1"]
& \prod_{v\in \Sigma}\pi_{1-2k}L_{K(1)}K(F_v)\arrow[r] \arrow[d, phantom, ""{coordinate, name=W}]
& \Hrm^1_\et(\Ocal_F[1/p],\Q_p/\Z_p(1-k))^\vee
\arrow[dll,rounded corners=8pt,curvarr=W] 
\\
\pi_{2-2k}L_{K(1)}K(\Ocal_F[1/p]) \arrow[r,"\Pi_2"]
& \prod_{v\in \Sigma}\pi_{2-2k}L_{K(1)}K(F_v)\arrow[r] 
& \Hrm^0_\et(\Ocal_F[1/p],\Q_p/\Z_p(1-k))^\vee\to 0,
\end{tikzcd}}
\end{center}

Since all terms in the exact sequence are finite, we obtain the result.	
\end{proof}

\subsection{A note on Hesselholt's expository note} According to Hesselholt, for $k>1$,
$$\left|L_p\left(\Q,\psi,\chi_\cyc^{-2k}\right)\right|_p=\frac{\#\pi_{4k-1}\fib\left(\tr_p^\wedge:K(\Z;\Z_p)\to\TC(\Z;\Z_p)\right)}{\#\pi_{4k}\fib\left(\tr_p^\wedge:K(\Z;\Z_p)\to\TC(\Z;\Z_p)\right)}$$
This formula is recovered for those values of $k$ in our result \ref{hessgeneralization}. The problem is that we would like to give the formula above in terms of the cyclotomic fiber for {\it all} $k$, instead of focusing on $k>1$. To do so, observe that
$$\left|L_p\left(F,\psi,\chi_\cyc^{n}\right)\right|_p=\frac{\# \Hrm^1_\et\left(\Ocal_F[1/p],\Q_p/\Z_p(n)\right)}{\#\Hrm^0_\et\left(\Ocal_F[1/p],\Q_p/\Z_p(n)\right)}=\frac{\# \Hrm^1_\et\left(\Ocal_F[1/p],\Q_p/\Z_p(n)\right)^\vee}{\#\Hrm^0_\et\left(\Ocal_F[1/p],\Q_p/\Z_p(n)\right)^\vee},$$
and let us also remember that, for all $n$,
$$\Hrm^1_\et\left(\Ocal_F[1/p],\Q_p/\Z_p(n)\right)^\vee\simeq\pi_{2n}\fib\left(L_{K(1)}\kappa:L_{K(1)}K(\Ocal_F[1/p])\to\prod_{v\in\Sigma}L_{K(1)}K(F_v)\right)$$
$$\Hrm^0_\et\left(\Ocal_F[1/p];\Q_p/\Z_p(n)\right)^\vee\simeq\pi_{1-2n}\fib\left(L_{K(1)}\kappa:L_{K(1)}K(\Ocal_F[1/p])\to\prod_{v\in\Sigma}L_{K(1)}K(F_v)\right)$$
On the other hand, the following diagram can be used to identify the fiber of $\kappa$ with the connective cover of the fiber of $\tr_p^\wedge$, where the equivalences are argued in \cite{Hess2018}:
\begin{align*}
\xymatrix{
\TC(\Z;\Z_p)\ar[r]^{\simeq } & \TC(\Z_p;\Z_p) \\
K(\Z;\Z_p) \ar[r]_{\kappa}\ar[u]^{\tr_p^\wedge} & K(\Z_p;\Z_p),\ar[u]_{\simeq \text{ in degrees }\geq0} 
}
\end{align*}
However, using Quillen's localization sequence will only give us information in degrees greater or equal than $2$, and that is why we will only have the formula for positive values of $k$. 

Instead of following that path, we realize we also have the following diagram, obtained after $K(1)$-localization, and substituting $\Z$ and $\Z_p$ by their respective generalizations to $F$, namely
\begin{align*}
\xymatrix@C=4em{
& L_{K(1)}\TC(\Ocal_F)\ar[r]^-{\simeq\text{ by \cite{Hess2018}}} & \prod_{v\in\Sigma}L_{K(1)}\TC(\Ocal_v) & \prod_{v\in\Sigma}L_{K(1)}K(\Ocal_v) \ar[l]_{\simeq}^{(\ast)}\\
\fib\left(L_{K(1)}\kappa\right)\ar@{-->}[r] & L_{K(1)}K(\Ocal_F) \ar[r]_-{L_{K(1)}\kappa}\ar[u]^{L_{K(1)}\tr} \ar[d]_{\simeq}^{\text{ by \autocite[Thm. 1.1]{bhatt2020remarks}}} & \prod_{v\in\Sigma}L_{K(1)}K(\Ocal_v)\ar[u]_{\prod_{v\in\Sigma}L_{K(1)}\tr} \ar@{=}@/_1pc/[ur] \ar[d]_{\simeq}^{\text{ by \autocite[Thm. 1.1]{bhatt2020remarks}}} & \\
\fib\left(L_{K(1)}\tr\right) \ar@{-->}[ur]& L_{K(1)}K(\Ocal_F[1/p]) \ar[r] & \prod_{v\in\Sigma}L_{K(1)}K(F_v) &
}
\end{align*}
where $(\ast)$ is an equivalence by \autocite[Cor. 1.3]{bhatt2020remarks}. With this, we obtain the desired equivalence
$$\fib\left(L_{K(1)}\tr\right)\simeq\fib(L_{K(1)}\kappa),$$
which, when combined with \ref{hessgeneralization}, gives the following result.
\begin{cor}\label{maincorollaryrealcase}
Let $n$ be an intger with $n\equiv 0\mod[F(\mu_p):F]$. Then, 
$$\left|L_p\left(F,\psi,\chi_\cyc^{-n}\right)\right|_p=\frac{\#\pi_{2n-1}\fib\left(L_{K(1)}\tr:L_{K(1)}K(\Ocal_F)\to L_{K(1)}\TC(\Ocal_F)\right)}{\#\pi_{2n}\fib\left(L_{K(1)}\tr:L_{K(1)}K(\Ocal_F)\to L_{K(1)}\TC(\Ocal_F)\right)}$$
\end{cor}

\newpage
%%% NCTR case
\section{The non-commutative and totally real case}\label{NCTR case}

\subsection{Initial considerations} As in the commutative case, let $F_\infty|F$ be an arbitrary admissible $p$-adic Lie extension -- cf. \ref{tradmissible} -- with Galois group $G$. Following \cite{Coates2003}, there exists a closed subgroup $H$ such that $\Gamma:=G/H\simeq\Z_p$. 

\begin{rem} 
For simplicity and without loss of generality, we assume that $G$ contains no element of order $p$. This makes most discussions much simpler, given that under this hypothesis $\Lambda(G)$ has finite global dimension, and $G$ also has finite $p$-homological dimension, which agrees with the dimension of $G$ as a $p$-adic Lie group (cf. \autocite[Remark 1.4]{OnmainconjecturesinnoncommutativeIwasawatheoryandrelatedconjectures}).
\end{rem}

Let us define the following sets, following \autocite[Section 2]{coates2004gl2}
$$\Sfrak=\left\{f\in\Lambda(G):\Lambda(G)/(f)\in\Mod_{\Lambda(H)}^\mathrm{fg}\right\},\qquad \Sfrak^\ast =\varinjlim_{n} p^n\Sfrak,$$
where $\Mod^{\mathrm{fg}}_{\Lambda(H)}$ is the category of finitely generated $\Lambda(H)$-modules. 

From now on, we let $\Sfrak'\in\{\Sfrak,\Sfrak^\ast\}$. Then, the following properties are known:
\begin{enumerate}
	\item \autocite[Proposition 2.3]{Coates2003} If $M$ is a finitely generated $\Lambda(G)$-module, then $M$ is finitely generated over $\Lambda(H)$ if and only if $M$ is $\Sfrak$-torsion. 
	\item \autocite[Theorem 2.4]{Coates2003} $\Sfrak'$ is multiplicatively closed, and is a (left or right) Ore set in $\Lambda(G)$. 
	\item \autocite[Theorem 2.4]{Coates2003} The elements of $\Sfrak'$ are non-zero divisors in $\Lambda(G)$. 
\end{enumerate}
Thus, one can define $\Mcal_{\Sfrak'}(G)$, the set of finitely generated $\Lambda(G)$-modules with the property that $\Lambda(G)_{\Sfrak'}\otimes_{\Lambda(G)} M$ vanishes or, equivalently, that are $\Sfrak'$-torsion.

\begin{rem}\label{generalizedIwasawa}
Generally, it is not known whether $X=X(F_\infty)$ lies in $\Mcal_{\Sfrak'}(G)$. For that reason, one works with extensions $F_\infty|F$ satisfying the {\it generalized Iwasawa conjecture}, meaning that there exists a finite extension $F'|F$ in $F_\infty$ such that $\Gal(F_\infty|F')$ is pro-$p$, and $X(F'^{\cyc})$ is a finitely generated $\Z_p$-module, or with the same hypotheses as Burns in \autocite[Theorem 9.1]{OnmainconjecturesinnoncommutativeIwasawatheoryandrelatedconjectures}. 
\end{rem}

\subsection{$K$-theoretic characteristic elements} From now on, we denote we use $\Lambda$ to refer to $\Lambda(G)$, unless it is otherwise stated. Before being able to state an Iwasawa main conjecture for the non-commutative case, the concept of a characteristic element must be introduced. To do so, recall that there is a localization sequence in $K$-theory of the form
$$\cdots\to K_1(\Lambda)\to K_1(\Lambda_{\Sfrak'})\xrightarrow{\partial_G} K_0(\Lambda,\Lambda_{\Sfrak'})\to K_0(\Lambda)\to K_0(\Lambda_{\Sfrak'})\to 0,$$
where $K_0(\Lambda,\Lambda_{\Sfrak'})$ is the 0-th relative $K$-theory group of $\Lambda\to\Lambda_{\Sfrak'}$. 
\begin{rem}
As it will be observed later, it is important to note that as explained in \autocite[\S1.1.2]{OnmainconjecturesinnoncommutativeIwasawatheoryandrelatedconjectures}, under the hypothesis that $G$ has no element of order $p$, there is a natural isomorphism between $K_0(\Lambda,\Lambda_{\Sfrak'})$ and the Grothendieck group $G_0(\Mcal_{\Sfrak'}(G))$. 	
\end{rem}

As mentioned in \autocite[Remark 1.1]{OnmainconjecturesinnoncommutativeIwasawatheoryandrelatedconjectures}, under our hypotheses, the map $\partial_G$ is always surjective. For this reason, we can provide the following definition.

\begin{dfn}
	Given $X\in K_0(\Lambda,\Lambda_{\Sfrak'})$, a {\it characteristic element} for $X$ is $\cfrak(X)\in K_1(\Lambda_{\Sfrak'})$ such that $\partial_G\cfrak(X)=X$.
\end{dfn}

\subsection{Evaluation at $K$-theoretic characteristic elements} This section is explained in more generality in \autocite[\S 1.1.3]{OnmainconjecturesinnoncommutativeIwasawatheoryandrelatedconjectures}. Of course, in order to link this definition to the point-wise property in the Iwasawa main conjectures, we need to define a way to evaluate the elements $\xi\in K_1(\Lambda_{\Sfrak'})$. To fulfill this analogy, let $\rho:G\to \mathrm{Aut}_{\Z_p}(\Gamma)$ be a (continuous) Artin representation of $G$. It is clear that $\rho$ extends canonically to a map $\Lambda_{\Sfrak'}\to\mathrm{End}_{\Z_p}(Q)$, where $Q=\mathrm{Frac}(\Lambda)$. If we then apply to it the functor $K_1$, we obtain a map
$$\rho_\ast:K_1(\Lambda_{\Sfrak'})\to Q^\times.$$
Localizing at the kernel $\kappa=\ker(\pi)$ of the projection map $\pi:\Lambda\to\Z_p$ produces a map $\pi_\kappa:\Lambda_\kappa\to\Q_p$. 

\begin{dfn} 
Following the notation from the previous paragraph, and if $\xi\in K_1(\Lambda_{\Sfrak'})$ and $\rho\in\Art(G)$, the {\it image of $\xi$ at $\rho$} is defined by
$$
\xi(\rho)=\begin{cases}
\pi_\kappa\left(\rho_\ast(\mathfrak{c})\right), & \rho_\ast(\xi)\in\Lambda_\kappa,\\
\infty & \text{otherwise.}	
\end{cases}
$$
\end{dfn}

\subsection{$p$-adic Artin $L$-functions}  During this subsection, we follow \autocite[\S9.1]{OnmainconjecturesinnoncommutativeIwasawatheoryandrelatedconjectures} almost exactly. Let $\rho$ be a finite-dimensional complex representation of any finite quotient of a compact $p$-adic Lie extension $F_\infty|F$ that is unramified outside a finite set of places $\Sigma$ of $F$. Denote by $\Art^+(G)$ the subset of $\Art(G)$ of representations that are trivial on the Galois group of some intermediate field $F'$ of $F_\infty|F$ that is totally real. Also, let $\Sigma'$ be a finite set of places of $F$ and let $L^{\Sigma'}(s,\rho)$ be the Artin $L$-function of $\rho$ truncated by removing the Euler factors of all non-archimedean places in $\Sigma'$. 

\begin{dfn} 
For each $\rho\in\Art^+(G)$, the {\it $\Sigma$-truncated $p$-adic Artin $L$-function} $L_p^\Sigma(s,\rho)$ is defined as the unique $p$-adic meromorphic function on $\Z_p$ such that for all integers $n<0$ and all field isomorphisms $j:\C_p\to\C$, 
$$L^{\Sigma_p}\left(n,(\rho\otimes \omega_F^{n-1})^j\right) = j\left(L_p^\Sigma(n,\rho)\right),$$
where $\Sigma_p$ is the union of $\Sigma$ and all places of $F$ that lie above $p$, and $\omega_F$ denotes the Teichm\"uller character. 	
\end{dfn}

\subsection{Twisting in the non-commutative setting} Let $X$ be a perfect complex of $\Lambda$-modules, and let $\rho:G\to\GL_n(\Ocal)$ be a continuous representation of $G$. Then, the complex $X(\rho):=\Ocal^n\otimes_{\Z_p}X$
can be endowed with a $G$-action by setting
$$g(a\otimes x):=\rho(g)(a)\otimes g(x),\qquad g\in G, a\in\Ocal^n, x\in X^j.$$

\subsection{Bockstein cohomology} As it is mentioned in \autocite[\S3.1]{burns2005leading}, given a bounded complex of $\Lambda$-modules $X$, the choice of a generator $\gamma$ of $\Gamma$ induces a morphism $\theta_\gamma$ in $D^\flat(\Lambda)$ of the form $X\to X[1]$. Given a continuous representation $\rho:G\to\GL_n(\Ocal)$ where $\Ocal$ denotes the valuation ring of some finite extension $L$ of $\Q_p$, $\theta_\gamma$ induces $T_\rho\otimes_{\Lambda}^\Lbf X\to T_\rho\otimes_{\Lambda}^\Lbf X[1]$. 
\begin{dfn}
The homomorphisms induced by $\theta_\gamma$ 
$$\beta_\nu:\Tor_\nu^{\Lambda}(T_\rho, X)\to\Tor_{\nu-1}^{\Lambda}(T_\rho, X)$$
are referred to as the {\it Bockstein homomorphisms} of $(X,T_\rho,\gamma)$. Moreover, these homomorphisms induce a complex, the {\it Bockstein complex associated to} $X\in D^\mathrm{perf}(\Lambda(G))$, and $X$ is said to be {\it semi-simple at $\rho$} the cohomology of the complex is $\Z_p$-torsion in each degree. Additionally, the {\it Bockstein cohomology groups} $\Hrm_\beta^i(G,X(\rho))$ are defined as the cohomology groups of the complex given by the homomorphisms
$$\beta_{-\nu}:\Tor_{-n}^{\Lambda}(T_\rho, X)\to\Tor_{-n-1}^{\Lambda}(T_\rho, X).$$
\end{dfn}

\subsection{The totally real non-commutative IMC} Following \autocite[\S 9.2]{OnmainconjecturesinnoncommutativeIwasawatheoryandrelatedconjectures}, given a $G$-module $M$, $M^\sharp$ is the $G_F$-module whose action by $\sigma\in G_F$ is given by $\chi(\sigma)\overline{\sigma}^{1}$, where $\overline{\sigma}$ is the image of $\sigma$ in $G$. We then define the derived complex 
$$C_{\infty,\Sigma}:=\R\Gamma_{\et, c}\left(\Spec(\Ocal_{F_\infty}[1/p], \Lambda^\sharp(1)\right)\in D^\flat (\Lambda)$$
of \'{e}tale cohomology with compact support, which is in fact perfect as stated in the citation. Moreover, according to [op. cit, \S 9.3], the $n$-th cohomology groups of $C_{\infty,\Sigma}$ can be canonically identified as 
$$
\Hrm^n(C_{\infty,\Sigma})\cong\begin{cases}
	X(F_\infty), & n = 2 \\
	\Z_p, & n = 3 \\
	0 & n\neq 2, 3. 
\end{cases}
$$
\begin{thm}
	\label{NCTRIMC} {\rm (\autocite[Thm 9.1]{OnmainconjecturesinnoncommutativeIwasawatheoryandrelatedconjectures})} There exists an element $\Lcal_p\in K_1(\Lambda_\Sfrak)$ such that
	\begin{enumerate}
		\item For all $\rho\in\Art(G)$, 
		$$\Lcal_p(\rho\otimes \chi) = L_p^{\Sigma}(1,\rho\otimes\chi)\mod\Q_p^\times,$$
		for almost all $\idbf\neq\chi\in\hat{\Gamma}$. 
		\item The composition of the boundary map $\partial_G:K_1(\Lambda_\Sfrak)\to K_0(\Lambda, \Lambda_{\Sfrak^\ast})$ with the isomorphism $K_0(\Lambda, \Lambda_{\Sfrak^\ast})\cong G_0(\Mcal_{\Sfrak^\ast}(G))$ sends $\Lcal_p\mapsto [X(F_\infty)]-[\Z_p]$. 
	\end{enumerate}
\end{thm}

\subsection{The Bockstein-cohomological Euler characteristic formula} Let $X\in\Mcal_{\Sfrak^\ast}(G)$ be a semi-simple complex at $\rho\in\mathrm{Art}(G)$, and such that $\Q_p\otimes_{\Z_p}A(\rho)$ is acyclic. Then, it follows from \autocite[Proposition 3.16]{burns2005leading} that there is a commutative diagram
$$
\xymatrix@C=3em@R=2em{
K_1(\Lambda, \Sigma_X)\ar[r]^-{\ch_{G,X}} \ar[d]_-{(-)(\rho)} & K_1(\Lambda_{\Sfrak^\ast}) \ar[d]^-{(-)(\rho)} \\
L^\times \ar@{=}[r] & L^\times,
}
$$
where the map $\ch_{G,X}$ and the subcategory $\Sigma_X$ of $C^\perf(\Lambda)$ is constructed in \autocite[\S 3.2.3]{burns2005leading}. For a detailed description of the map, we refer the reader to \autocite[Appendix A.]{burns2011descent}

The following statement is the non-commutative counterpart of Corollary \ref{maincorollaryrealcase}, that is, the main result of the previous section. 

\begin{cor} \label{BurnsEulerCharFla}
	{\rm (\autocite[Proposition 3.19]{burns2005leading} corrected using \autocite[Appendix C.]{burns2011descent})} If $X$ is semi-simple at $\rho$, then for any $\xi\in\mathrm{Im}(\ch_{G,X})\leq K_1(\Lambda_{\Sfrak^\ast})$, 
	$$\left|\xi(\rho)\right|_p^{[L:\Q_p]}=\prod_{\nu\in\Z}\sharp\Hrm^i_\beta\left(G,X(\mathrm{Ad}(\rho))\right)^{(-1)^\nu}.$$
\end{cor}

\begin{rem}
	In the classical, commutative case, one has that $G$ is of rank one, and $\rho=\chi_\cyc$. Given a finitely generated and torsion $\Lambda$-module $M$, the differentials induced by the Hochschild-Serre spectral sequence 
$$
\Hrm^i(G,M^\vee)\to\Hrm^{i+1}(G,M^\vee) 
$$
are manipulated in order to obtain the equality in \ref{etale_formula}. According to the result above, however, one only needs to use a projective resolution $P$ of $M^\vee$ to obtain the same formula, or simply observe that
$$(d^i)^\vee = \beta_{i+1}:\Tor_{i+1}^{\Lambda}(\Z_p,P)\to\Tor_i^{\Lambda}(\Z_p,P).$$
\end{rem}

\begin{cor} For all $\rho\in\Art(G)$ and almost all $\idbf\neq \chi\in\hat{\Gamma}$, 
	$$\left|\Lcal_p(\rho\otimes \chi)\right|_p =\left|L_p^{\Sigma}(1,\rho\otimes\chi)\right|_p =\prod_{\nu\in\Z}\sharp\Hrm^i_\beta\left(G, C_{\infty,\Sigma}\left(\mathrm{Ad}(\rho\otimes\chi)\right)\right)^{(-1)^\nu}.$$
\end{cor}

\begin{proof}
The proof of this result follows directly from combining the IMC in the non-commutative setting \ref{NCTRIMC} together with \ref{BurnsEulerCharFla}.	
\end{proof}

\newpage
%% Print bibliography
\printbibliography

\end{document}